\documentclass{amsart}

\usepackage{amsmath}
\usepackage{amscd}
\usepackage{verbatim}
\usepackage{amsfonts}
\usepackage{amsthm}
\usepackage{amssymb}
\usepackage{graphicx}
\usepackage{epsfig}
\usepackage[cmtip,all]{xy}
\usepackage{url}
\usepackage{diagrams}

\newtheorem{theorem}{Theorem}[section]

\newtheorem{corollary}[theorem]{Corollary}
\newtheorem{proposition}[theorem]{Proposition}

\theoremstyle{definition}
\newtheorem{definition}[theorem]{Definition}
\newtheorem{example}[theorem]{Example}

\theoremstyle{remark}
\newtheorem{remark}[theorem]{Remark}

\numberwithin{equation}{section}

\begin{document}

\title{Extended quandle spaces and shadow homotopy invariants of classical links}
\author{SEUNG YEOP YANG}

\address{}
\date{}

\maketitle
\markboth{S. Y. Yang}{\hfil{\sc Extended quandle spaces and shadow homotopy invariants of classical links}\hfil}


\begin{abstract}
In 1993, Fenn, Rourke and Sanderson introduced rack spaces and rack homotopy invariants, and modifications to quandle spaces and quandle homotopy invariants were introduced by Nosaka in 2011. In this paper, we define the Cayley-type graph and the extended quandle space of a quandle in analogy to rack and quandle spaces. Moreover, we construct the shadow homotopy invariant of a classical link and prove that the shadow homotopy invariant is equal to the quandle homotopy invariant multiplied by the order of a quandle.
\end{abstract}

\section{Introduction}

\subsection{Preliminaries}
The algebraic structure $(X,*)$ with a set $X$ and a binary operation $*:X \times X \rightarrow X$ is said to be a \emph{magma}. Let us consider the following properties of $*:$
\begin{enumerate}
  \item (Right self-distributivity) For any $a,b,c \in X,$ $(a*b)*c=(a*c)*(b*c).$
  \item (Invertibility) For each $b \in X$, the map $*_{b}:X \rightarrow X$ defined by $*_{b}(x)=x*b$ is invertible.
  \item (Idempotency) For any $a \in X,$ $a*a=a.$
\end{enumerate}
A magma is called a \emph{right distributive structure} or simply \emph{RDS} if the binary operation satisfies the right self-distributive property. If the binary operation satisfies right self-distributivity and invertibility, then the magma is said to be a \emph{rack} \cite{CW59,FenRou92}. A \emph{quandle} (introduced by Joyce \cite{Joy82} and Matveev \cite{Mat82}) is a magma that satisfies all three properties above. The three properties are motivated by the three Reidemeister moves, and quandles can be used to classify classical knots. See \cite{Joy82} and \cite{Mat82} for details.

\begin{example}
The following are some examples of racks and quandles.
\begin{enumerate}
  \item The residues modulo $n,$ $C_{n}=\{0,1,\ldots,n-1\}$ with the operation $i*j=i+1$ mod $n$ for all $i,j \in C_{n}$ is called the \emph{cyclic rack} of order $n.$
  \item A set $X$ with the binary operation $*_{0}$ satisfying $a *_{0} b = a$ for any $a,b \in X$ is said to be a \emph{trivial quandle}.
  \item The cyclic group $\mathbb{Z}_{n}$ of order $n$ with the operation $i*j=2j-i$ for all $i,j \in \mathbb{Z}_{n}$ forms a quandle denoted by $R_{n},$ and it is called a \emph{dihedral quandle} of order $n.$
  \item A module $M$ over the Laurent polynomial ring $\mathbb{Z}[t^{\pm1}]$ with the operation $a * b = ta + (1-t)b$ forms a quandle structure called an \emph{Alexander quandle}.
\end{enumerate}
\end{example}

A map $f:X \rightarrow Y$ between two quandles $(X,*)$ and $(Y,*')$ is called a \emph{quandle homomorphism} if $f(a * b)=f(a)*'f(b)$ for all $a, b \in X$. A \emph{quandle isomorphism} is a bijective quandle homomorphism. A quandle isomorphism from a quandle $X$ onto itself is said to be a \emph{quandle automorphism}. Notice that the map $*_{b}$ in the second axiom of a quandle is a quandle automorphism. For a quandle $X,$ a \emph{quandle inner automorphism group} denoted by $\textrm{Inn}(X)$ is a subgroup of the quandle automorphism group $\textrm{Aut}(X)$ generated by the maps $*_{b}$ for all $b \in X.$ If the map $i:X \rightarrow \textrm{Inn}(X)$ defined by $i(x)=*_{x}$ is injective, then a quandle $X$ is called \emph{faithful}. A quandle $X$ is said to be \emph{connected} (or \emph{homogeneous}, respectively) if the action of $\textrm{Inn}(X)$ (or $\textrm{Aut}(X)$, respectively) on $X$ is transitive.

Let $D_{L}$ be an oriented link diagram of an oriented classical link $L,$ and $\mathcal{R}$ be the set of all arcs of $D_{L}.$ Given a quandle $X,$ the map $\mathcal{C}:\mathcal{R} \rightarrow X$ such that the relation depicted in Figure \ref{quandleshadowcoloring}(i) holds at each crossing of $D_{L}$ is said to be a \emph{quandle coloring} of $D_{L}$ by $X.$ In analogy to the Wirtinger presentation of the link group $\pi_{1}(\mathbb{R}^{3} \setminus L)$ for an oriented link $L,$ the link quandle $Q_{L}$ was introduced in \cite{Joy82} and \cite{Mat82}. A quandle coloring of $D_{L}$ by $X$ can be considered as a quandle homomorphism from $Q_{L}$ to $X$ because there is a one-to-one correspondence between two sets $\textrm{Hom}(Q_{L},X)$ and $\textrm{Col}_{X}(D_{L})$ where $\textrm{Col}_{X}(D_{L})$ is the set of quandle colorings of $D_{L}$ by $X$ (see \cite{Joy82} for details). Every quandle coloring can be extended to a shadow coloring introduced in \cite{FenRouSan04} and \cite{CarKamSai01}. Let $\widetilde{\mathcal{R}}$ denote the set of arcs in $D_{L}$ and regions separated by immersed plane curves of $D_{L}$ (i.e. without height relation information). A \emph{shadow coloring} of a link diagram $D_{L}$ by $X$ is a map $\widetilde{\mathcal{C}}:\widetilde{\mathcal{R}} \rightarrow X$ satisfying the relation depicted in Figure \ref{quandleshadowcoloring}(ii) holds at each crossing of $D_{L}.$ We denote the set of shadow colorings of a link diagram $D_{L}$ by a given quandle $X$ by $\textrm{SCol}_{X}(D_{L}).$ It is well-known that cardinalities $|\textrm{Col}_{X}(D_{L})|$ and $|\textrm{SCol}_{X}(D_{L})|$ do not depend on the choice of link diagram that means they are link invariants. We denote them by $|\textrm{Col}_{X}(L)|$ and $|\textrm{SCol}_{X}(L)|,$ respectively. Note that if $X$ is finite, then $|\textrm{SCol}_{X}(D_{L})|=|X||\textrm{Col}_{X}(D_{L})|.$ More generally, the choice of color for any region of $\mathbb{R}^{2}\setminus D_{L}$ yields uniquely the shadow extension. See \cite{CarKamSai04} for further details.

\begin{figure}[h]
\centerline{{\psfig{figure=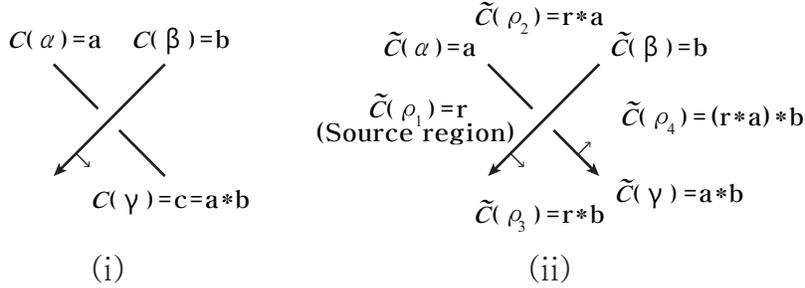,height=3.9cm}}}
\caption{Quandle coloring and shadow coloring relations at each crossing.}
\label{quandleshadowcoloring}
\end{figure}

\subsection{Historical background}
The first homology theory, called rack homology, using right distributive structures and its geometric realizations, called rack space and extended rack space, were introduced by Fenn, Rourke and Sanderson \cite{FenRouSan93,FenRouSan95}. The recent paper by Fenn \cite{Fen12} states: ``Unusually in the history of mathematics, the discovery of the homology and classifying space of a rack can be precisely dated to 2 April 1990." Przytycki \cite{Prz11,Prz15} constructed one-term and multi-term distributive homology theories as generalizations of Fenn, Rourke and Sanderson's studies. The rack homology theory was modified to so-called quandle homology by Carter, Jelsovsky, Kamada, Langford and Saito\cite{CarJelKamLanSai03} in order to define cocycle invariants of classical knots. Later, Nosaka \cite{Nos11} constructed quandle spaces\footnote{Homotopy groups of rack spaces and quandle spaces were studied in \cite{FenRouSan04,FenRouSan07} and \cite{Nos11,Nos13}, respectively.} and quandle homotopy invariants of classical links by modifying rack spaces and rack homotopy invariants of Fenn, Rourke and Sanderson \cite{FenRouSan93}.

For a quandle $X,$ we construct the Cayley-type graph of $X.$ In analogy to the quandle space, we also construct the extended quandle space of $X$ modifying the extended rack space in \cite{FenRouSan95}. The extended quandle space is a geometric realization of quandle homology introduced in \cite{CarJelKamLanSai03}. More precisely, it is a geometric realization of the pre-cubic set related to quandle homology. Furthermore, we define the shadow homotopy invariant of a classical link by using the extended quandle space (or the action quandle space introduced in \cite{Nos13}) and prove that for a finite connected quandle $X$ the shadow homotopy invariant is equal to the quandle homotopy invariant multiplied by $|X|$ in the integral group ring $\mathbb{Z}[\pi_{2}(B^{Q}X)]$ (see Section \ref{Section 3}); compare with a classical observation that $|\textrm{SCol}_{X}(L)|=|X||\textrm{Col}_{X}(L)|$ for a finite quandle $X.$

\section{The geometric realization of a distributive structure homology}

A \emph{pre-simplicial set}\footnote{Eilenberg and Zilber \cite{EZ50} introduced the concept of a simplicial set under the name \emph{complete semi-simplicial complex}, and their semi-simplicial complex is now usually called a \emph{pre-simplicial set} \cite{Lod98,May67}. See \cite{Prz15} for details.} $\mathcal{X}=(X_{n},d_{i})$ is a collection of sets $X_{n}$ $(n \geq 0)$ together with face maps $d_{i}:=d_{i,n}:X_{n} \rightarrow X_{n-1}$ $(0 \leq i \leq n)$ satisfying $d_{i}d_{j}=d_{j-1}d_{i}$ for $i < j.$ Let $C_{n}=\mathbb{Z}X_{n}$ be the free abelian group generated by elements of $X_{n}$ and let $\partial_{n}=\sum\limits_{i=0}^{n}(-1)^{i}d_{i}.$ Then $(C_{n},\partial_{n})$ forms a chain complex, so we can define homology groups of it. An important example of the above construction is group homology.

The geometric realization of a pre-simplicial set can be obtained by gluing together standard simplices with an instruction provided by the pre-simplicial set. Let $\triangle^{n}=\{(t_{0},\ldots,t_{n}) \in \mathbb{R}^{n+1} \mid \sum\limits_{i=0}^{n}t_{i}=1, t_{i} \geq 0\}$ be the standard $n$-simplex. For each $0 \leq i \leq n,$ we consider the coface map $d^{i}:=d^{i,n}:\triangle^{n-1} \rightarrow \triangle^{n}$ defined by $d^{i}(t_{0}, \ldots, t_{n-1})=(t_{0},\ldots ,t_{i-1},0,t_{i},\ldots,t_{n-1}).$ They satisfy relations, $d^{i}d^{j-1}=d^{j}d^{i}$ for $0 \leq i < j \leq n,$ and the geometric realization $|\mathcal{X}|$ of $\mathcal{X}$ is defined as the quotient of the disjoint union $\coprod\limits_{n}(X_{n} \times \triangle^{n})$ by $(\textbf{x}, d^{i}(\textbf{t})) \sim (d_{i}(\textbf{x}),\textbf{t})$ where $X_{n}$ has discrete topology, $\textbf{x} \in X_{n}$ and $\textbf{t} \in \triangle^{n-1}.$ Note that $|\mathcal{X}|$ is a CW-complex with one $n$-cell for each element of $X_{n}.$

A pre-cubic set\footnote{Jean-Pierre Serre was the first to consider cubic category in his PhD thesis \cite{Ser51}. Extensive review of cubic category can be found in \cite{BHS11}.} and its geometric realization can be defined similarly. We review the theory of pre-cubic sets and geometric realizations based on \cite{FenRouSan95}. A \emph{pre-cubic set} $\mathcal{X}=(X_{n},d_{i}^{\varepsilon})$ consists of a collection of sets $X_{n}$ for $n \geq 0$ and face maps $d_{i}^{\varepsilon}:=d_{i,n}^{\varepsilon}:X_{n} \rightarrow X_{n-1}$ for $1 \leq i \leq n$ and $\varepsilon \in \{0,1\}$ satisfying $d_{i}^{\varepsilon}d_{j}^{\delta}=d_{j-1}^{\delta}d_{i}^{\varepsilon}$ if $i < j$ for $\delta,\varepsilon \in \{0,1\}.$ We let $C_{n}=\mathbb{Z}X_{n}$ and $\partial_{n}=\sum\limits_{i=1}^{n}(-1)^{i+1}(d_{i}^{0}-d_{i}^{1})$ in order to obtain chain complex $(C_{n},\partial_{n})$ and its homology groups.

The geometric realization $|\mathcal{X}|$ of a pre-cubic set $\mathcal{X}$ is a CW-complex defined as the quotient of the disjoint union $\coprod\limits_{n}(X_{n} \times I^{n})$ by $(\textbf{x}, d^{i}_{\varepsilon}(\textbf{t})) \sim (d_{i}^{\varepsilon}(\textbf{x}),\textbf{t}),$ where $I^{n}=[0,1]^{n} \subset \mathbb{R}^{n}$ $(n \geq 0)$ and $d^{i}_{\varepsilon}:=d^{i,n}_{\varepsilon}:I^{n-1} \rightarrow I^{n}$ are maps defined by $d^{i}_{\varepsilon}(t_{1},\ldots,t_{n-1})=(t_{1},\ldots,t_{i-1},\varepsilon,t_{i},\ldots,t_{n-1})$ for $1 \leq i \leq n.$ Notice that the maps $d^{i}_{\varepsilon}$ satisfy the following relations:
$d^{i}_{\varepsilon}d^{j-1}_{\delta}=d^{j}_{\delta}d^{i}_{\varepsilon}, ~ 1 \leq i < j \leq n, ~ \delta, \varepsilon \in \{0,1\}.$

\subsection{Cayley-type graphs and extended quandle spaces of quandles}

In this section, we construct extended quandle spaces. We modify extended rack spaces introduced in \cite{FenRouSan95} in analogy to the way quandle spaces in \cite{Nos11} where obtained from rack spaces in \cite{FenRouSan93}. Since extended rack and quandle spaces are geometric realizations of homology theories of racks in \cite{FenRouSan95} and quandles in \cite{CarJelKamLanSai03}, respectively, we first review definitions of rack and quandle homologies based on \cite{CarKamSai04}.

Let $C_{n}^{R}(X)$ be the free abelian group generated by n-tuples $(x_{1}, \ldots ,x_{n})$ of elements of a rack $X$, i.e. $C_{n}^{R}(X)=\mathbb{Z}X^{n}=(\mathbb{Z}X)^{\otimes n}$. Define a boundary homomorphism $\partial_{n}:C_{n}^{R}(X) \rightarrow C_{n-1}^{R}(X)$ by for $n \geq 2$
$$\partial_{n} (x_{1}, \ldots , x_{n})=\sum\limits_{i=2}^{n}(-1)^{i}(d_{i}^{(*_{0})}-d_{i}^{(*)})(x_{1}, \ldots , x_{n})$$
where for $2 \leq i \leq n,$ $d_{i}^{(*_{0})}(x_{1}, \ldots , x_{n})=(x_{1},\ldots,x_{i-1},x_{i+1},\ldots,x_{n})$ and
$$d_{i}^{(*)}(x_{1}, \ldots , x_{n})=(x_{1}*x_{i},\ldots,x_{i-1}*x_{i},x_{i+1},\ldots,x_{n}),$$
and $\partial_{n}=0$ for $n \leq 1.$ Then $(C_{n}^{R}(X),\partial_{n})$ forms a chain complex, and it is called the \emph{rack chain complex} of $X.$

For a quandle $X,$ consider the subgroup $C_{n}^{D}(X)$ of $C_{n}^{R}(X)$ generated by n-tuples $(x_{1}, \ldots ,x_{n})$ of elements of $X$ with $x_{i}=x_{i+1}$ for some $i\in \{1,\ldots,n-1\}$ if $n \geq 2$, otherwise let $C_{n}^{D}(X)=0.$ Then $(C_{n}^{D}(X),\partial_{n})$ is a sub-chain complex of the rack chain complex $(C_{n}^{R}(X),\partial_{n}),$ so the quotient chain complex $(C_{n}^{Q}(X)=C_{n}^{R}(X)/C_{n}^{D}(X),\partial_{n}^{'})$ can be obtained where $\partial_{n}^{'}$ is the induced homomorphism. Hereafter, we denote all boundary maps by $\partial_{n}.$

For an abelian group $G,$ we define the chain and cochain complexes
$$C_{*}^{W}(X;G)=C_{*}^{W}(X) \otimes G, ~ \partial = \partial \otimes \textrm{Id};$$
$$C^{*}_{W}(X;G)=\textrm{Hom}(C_{*}^{W}(X), G), ~ \delta = \textrm{Hom}(\partial, \textrm{Id})$$
for W=R, D and Q. Then the $n$th \emph{rack, degenerate and quandle homology groups} and the $n$th \emph{rack, degenerate and quandle cohomology groups} of a rack / quandle $X$ with coefficient group $G$ are respectively defined as
$$H_{n}^{W}(X;G)=H_{n}(C_{*}^{W}(X;G)),~ H^{n}_{W}(X;G)=H^{n}(C^{*}_{W}(X;G))$$
for W=R, D and Q.

\begin{definition}\cite{FenRouSan93,FenRouSan95}
Let $X$ be a rack, and let $d_{i}^{(*_{0})},d_{i}^{(*)}$ be the face maps defined for $2 \leq i \leq n$ in the definition of the rack chain complex of $X.$
\begin{enumerate}
  \item The \emph{rack space} of $X,$ denoted by $BX,$ is the geometric realization of the pre-cubic set $(X_{n},d_{i}^{0},d_{i}^{1})=(X^{n},d_{i}^{(*_{0})},d_{i}^{(*)})$ where $X^{0}$ is a singleton set and $d_{1}^{(*_{0})}(x_{1}, \ldots , x_{n})=(x_{2},\ldots,x_{n})=d_{1}^{(*)}(x_{1}, \ldots , x_{n}).$
  \item The \emph{extended rack space} of $X,$ denoted by $B_{X}X,$ is the geometric realization of the pre-cubic set $(X_{n},d_{i}^{0},d_{i}^{1})=(X^{n+1},d_{i+1}^{(*_{0})},d_{i+1}^{(*)}).$
\end{enumerate}
\end{definition}

\begin{figure}[h]
\centerline{{\psfig{figure=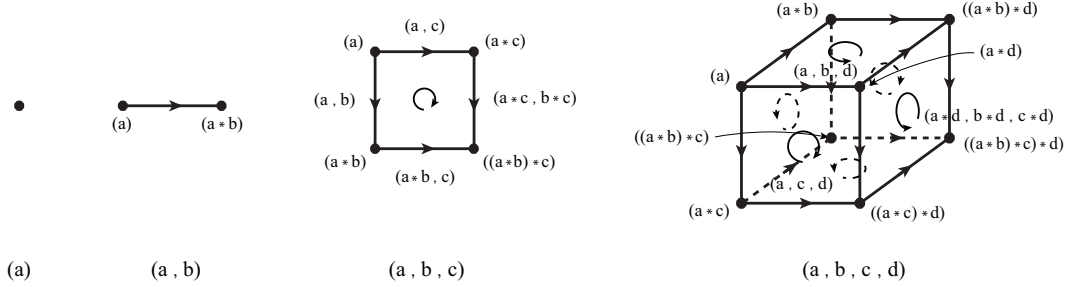,height=5.2cm}}}
\caption{Low-dimensional cells of an extended rack space.}
\label{extendedrackspace}
\end{figure}

For a quandle $X,$ Nosaka \cite{Nos11} constructed the quandle space $B^{Q}X$ of $X$ from the rack space $BX$ of $X$ by attaching extra cells. Extended quandle spaces can be obtained from extended rack spaces in a similar way. We consider the subspace $\coprod\limits_{n \geq 1}(\mathcal{D}^{n+1} \times I^{n})$ of $\coprod\limits_{n \geq 0}(X^{n+1} \times I^{n})$ where $\mathcal{D}^{n}=\{(x_{1},\ldots,x_{n}) \in X^{n} \mid x_{i}=x_{i+1} \textrm{~for some~} 1 \leq i \leq n-1\}.$ For each ${\bf x} \in \mathcal{D}^{n+1},$ consider the inclusion map $f_{\bf x}:((\{ {\bf x} \} \times I^{n})/\sim) \hookrightarrow B_{X}X$ where $({\bf x}, d^{i}_{\varepsilon}({\bf t})) \sim (d_{i}^{\varepsilon}({\bf x}),{\bf t}).$ We denote the mapping cone of $f_{\bf x}$ by $C_{f_{\bf x}}.$

\begin{definition}
The \emph{extended quandle space} of a quandle $X,$ denoted by $B_{X}^{Q}X$, is the union of mapping cones $\bigcup\limits_{{\bf x} \in \mathbb{D}} C_{f_{\bf x}}$ where $\mathbb{D}=\bigcup\limits_{k \geq 2}\mathcal{D}^{k}.$
\end{definition}

\begin{remark}
If we restrict $\mathcal{D}^{n}$ in the above construction to be $\mathcal{D}^{n}=\{(x_{1},\ldots,x_{n}) \in X^{n} \mid x_{i}=x_{i+1} \textrm{~for some~} 2 \leq i \leq n-1\}$ (i.e. $\mathcal{D}^{2}=\emptyset$), then the resulting space coincides with the space $BX_{Q}^{X}$ introduced in \cite{Nos13}. We call it the \emph{action quandle space} of a quandle $X.$
\end{remark}

We give a detailed description of the $3$-skeleton of an extended quandle space, denoted by $B_{X}^{Q}X^{3}.$ Let $B_{X}X^{3}$ be the $3$-skeleton of the extended rack space of a quandle $X.$ Each edge $e_{(a,a)}$ labeled by $(a,a)$ of $B_{X}X^{3}$ is a loop. We first glue a $2$-cell $D^{2}$ to $B_{X}X^{3}$ by a homeomorphism $\partial(D^{2}) \rightarrow e_{(a,a)}$ for each $(a,a) \in \mathcal{D}^{2}.$ We denote the space constructed by $\widetilde{B}_{X}X^{3}$ whose $2$-skeleton is the $2$-skeleton of $B_{X}^{Q}X^{3}.$ To build the $3$-skeleton $B_{X}^{Q}X^{3}$ we attach $3$-cells to its $2$-skeleton. We have two cases to consider.\\
\hspace*{10pt}\textbf{Case} $1.$ $(a,b,b) \in \mathcal{D}^{3}.$\\
Each $2$-cell $e_{(a,b,b)}$ labeled by $(a,b,b)$ of $\widetilde{B}_{X}X^{3}$ forms a sphere. We attach a $3$-cell $D^{3}$ to $\widetilde{B}_{X}X^{3}$ via a homeomorphism $\partial(D^{3}) \rightarrow e_{(a,b,b)}$ for each $(a,b,b) \in \mathcal{D}^{3}.$ See Figure \ref{extendedquandlespace}(i).\\
\hspace*{10pt}\textbf{Case} $2.$ $(a,a,b) \in \mathcal{D}^{3}.$\\
Each $2$-cell $e_{(a,a,b)}$ labeled by $(a,a,b)$ of $\widetilde{B}_{X}X^{3}$ is the side of a cylinder, and together with two disks, denoted by $D_{(a,a)}$ and $D_{(a*b,a*b)},$ whose boundaries are $e_{(a,a)}$ and $e_{(a*b,a*b)}$ it forms a cylinder. We attach a $3$-cell $D^{3}$ to $\widetilde{B}_{X}X^{3}$ via a homeomorphism $\partial(D^{3}) \rightarrow e_{(a,a,b)} \cup D_{(a,a)} \cup D_{(a*b,a*b)}$ for each $(a,a,b) \in \mathcal{D}^{3}.$ See Figure \ref{extendedquandlespace}(ii).\\
The space obtained from the above construction is the $3$-skeleton $B_{X}^{Q}X^{3}$ of the extended quandle space of $X.$

\begin{figure}[h]
\centerline{{\psfig{figure=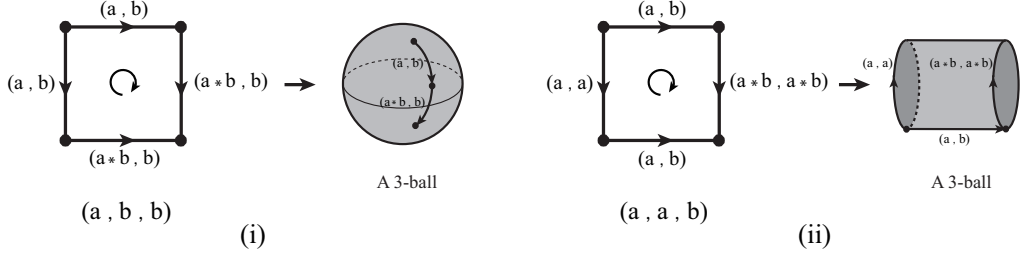,height=3.4cm}}}
\caption{Extra $3$-cells of an extended quandle space.}
\label{extendedquandlespace}
\end{figure}

\begin{remark}
In the above construction, if we attach only a $3$-cell $D^{3}$ to the $3$-skeleton of the extended rack space $B_{X}X^{3}$ for each $(a,b,b) \in \mathcal{D}^{3},$ then the resulting space is the $3$-skeleton of the action quandle space.
\end{remark}

It follows from the above construction that for an abelian group $G,$ the $i$-dimensional rack (quandle, respectively) homology group of $X$ is isomorphic to the $(i-1)$-dimensional homology group of the extended rack (quandle, respectively) space of $X,$ i.e. for each $i \geq 1,$
$$H_{i-1}(B_{X}X;G) \cong H_{i}^{R}(X;G) ,~ H_{i-1}(B_{X}^{Q}X;G) \cong H_{i}^{Q}(X;G).$$

In particular, the $1$-skeleton $B_{X}X^{1}$ of the extended rack space of a rack $X$ is called the \emph{rack graph} of $X.$ For a quandle $X,$ the \emph{quandle graph} of $X,$ denoted by $B_{X}^{Q}X^{1},$ is the graph obtained by deleting loops that labeled by $(a,a)$ for all $a \in X.$

\begin{figure}[h]
\centerline{{\psfig{figure=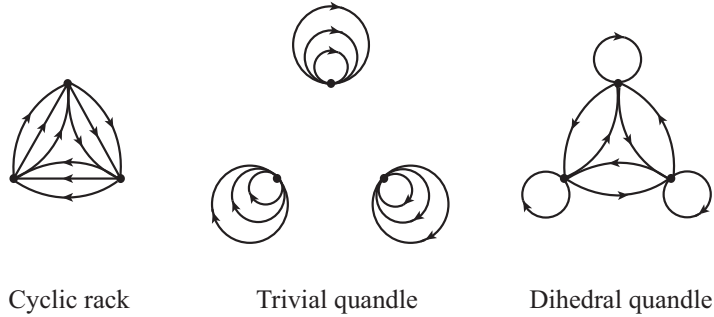,height=4.2cm}}}
\caption{Rack graphs of racks of order $3.$}
\label{extendedquandlegraph}
\end{figure}

Notice that every rack space and quandle space is path-connected, but not every extended rack space and extended quandle space is path-connected.

\begin{proposition}\label{Proposition 2.3}
Let $X$ be a quandle. Then the following are equivalent.
\begin{enumerate}
  \item $X$ is a connected quandle.
  \item The rack graph (or quandle graph) of $X$ is a connected graph.
  \item The extended rack space (or extended quandle space) of $X$ is path-connected.
\end{enumerate}
Moreover, the number of connected components of the rack graph (or quandle graph) of $X$ is equal to $\emph{rank}H_{1}^{R}(X)(=\emph{rank}H_{1}^{Q}(X)).$
\end{proposition}

\begin{proof}
$(1)\Leftrightarrow(2)$ A quandle $X$ is connected if and only if for any $a,b \in X,$ there exist elements $x_{1},\ldots,x_{k} \in X$ such that $(\cdots((a*x_{1})*x_{2})*\cdots)*x_{k}=b.$ It is equivalent to that there exist edges labeled by $(a,x_{1}),(a*x_{1},x_{2}),\ldots,((\cdots((a*x_{1})*x_{2})*\cdots)*x_{k-1},x_{k})$ such that they connect two points $(a)$ and $(b)$ in the rack graph (or quandle graph) of $X.$\\
$(2)\Rightarrow(3)$ It is obvious since the extended rack space (or extended quandle space) of $X$ is a CW-complex and the rack graph of $X$ is its $1$-skeleton.\\
$(3)\Rightarrow(1)$ Consider the action of $\textrm{Inn}(X)$ on $X.$ Notice that for each cell the labels of all vertices of the cell belong to the same orbit (more precisely, they belong to the orbit of the first coordinate of the label of the cell). Therefore, if $X$ is not connected, then the extended rack space (or extended quandle space) of $X$ is not path-connected.

Moreover, since $H_{0}(B_{X}X^{1}) \cong H_{1}^{R}(X) \cong H_{1}^{Q}(X),$ $\textrm{rank}H_{1}^{R}(X)=\textrm{rank}H_{1}^{Q}(X)$ is the number of components of the graph $B_{X}X^{1}.$
\end{proof}

Eisermann \cite{Eis03} showed that the second quandle homology detects the unknot.

\begin{theorem}\cite{Eis03}\label{Theorem 2.4}
Let $K$ be a knot, and let $Q_{K}$ be its knot quandle. If $K$ is trivial, then $H_{2}^{Q}(Q_{K})=0.$ If $K$ is non-trivial, however, then $H_{2}^{Q}(Q_{K})=\mathbb{Z}.$
\end{theorem}

The following corollary can be obtained from the above theorem.

\begin{corollary}
Let $K$ be a knot, and let $B_{Q_{K}}^{Q}{Q_{K}}^{2}$ be the $2$-skeleton of the extended quandle space of the knot quandle $Q_{K}.$ Then $K$ is the unknot if and only if $\pi_{1}(B_{Q_{K}}^{Q}{Q_{K}}^{2})$ is a perfect group.
\end{corollary}

\begin{proof}
Since every knot quandle is connected, $B_{Q_{K}}^{Q}Q_{K}$ is path-connected by Proposition \ref{Proposition 2.3}. Notice that $H_{1}(B_{Q_{K}}^{Q}Q_{K}) \cong H_{2}^{Q}(Q_{K}).$ Thus $K$ is trivial if and only if $H_{1}(B_{Q_{K}}^{Q}Q_{K})=0$ by Theorem \ref{Theorem 2.4}. Since $\pi_{1}(X)$ depends only on the $2$-skeleton of $X$ for any CW-complex $X,$ $H_{1}(B_{Q_{K}}^{Q}Q_{K})=0$ if and only if $\pi_{1}(B_{Q_{K}}^{Q}{Q_{K}}^{2})$ is a perfect group as desired.
\end{proof}

\section{Shadow homotopy invariants of classical links}\label{Section 3}

The rack homotopy invariant of a framed oriented link was introduced by Fenn, Rourke and Sanderson \cite{FenRouSan93}, and it was modified by Nosaka \cite{Nos11} to construct the quandle homotopy invariant of an oriented link. By using an action quandle space (or an extended quandle space) and a shadow coloring, we obtain the shadow homotopy invariant of an oriented link in a similar manner.

Let $L$ be an oriented link, and let $D_{L}$ be its diagram in $I^{2}=[0,1]\times[0,1].$ See Figure \ref{shadowhomotopyinvariant}. We then construct a decomposition of $I^{2}$ with respect to $D_{L}$ in the following way. We choose a point in each region separated by immersed plane curves of $D_{L},$ except the region, called the \emph{$\infty$-region}, which is adjacent to the boundary $\partial I^{2}$ of $I^{2}.$ If $R_{1}$ and $R_{2}$ are two regions separated by an arc $\alpha,$ then we connect points $p_{1}$ and $p_{2}$ which are corresponding to the regions $R_{1}$ and $R_{2}$ respectively by a directed edge $e$ as depicted in Figure \ref{decomposition}. Especially, if one of $R_{1}$ and $R_{2}$ is the $\infty$-region, say $R_{2}$ is the $\infty$-region without loss of generality, then we connect the point $p_{1}$ to the boundary $\partial I^{2}$ of $I^{2}$ as depicted in Figure \ref{decomposition}.

\begin{figure}[h]
\centerline{{\psfig{figure=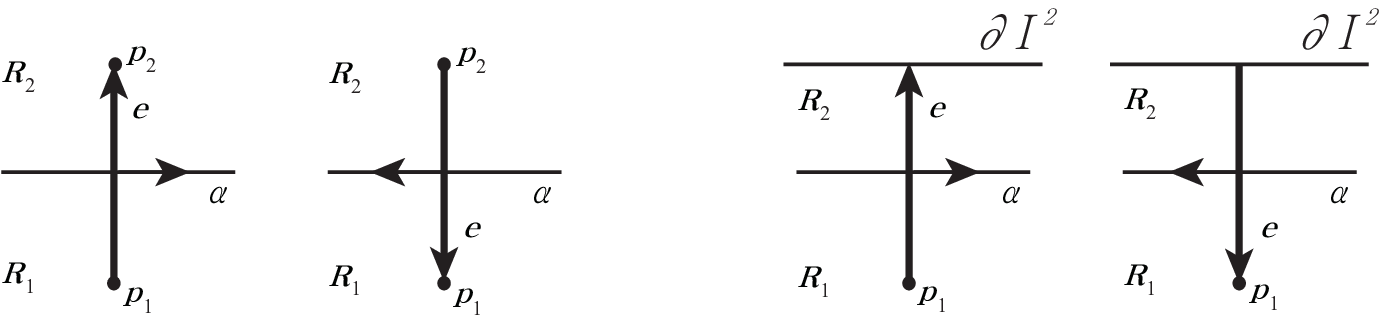,height=2.5cm}}}
\caption{A decomposition of $I^{2}$ with respect to $D_{L}$.}
\label{decomposition}
\end{figure}

Given a shadow coloring $\widetilde{\mathcal{C}}$ of $D_{L}$ by a quandle $X,$ we label a decomposition of $I^{2}$ by $D_{L}$ illustrated above in order to define a map from $I^{2}$ to the action quandle space $BX_{Q}^{X}$ (or the extended quandle space $B_{X}^{Q}X$). We first label the boundary $\partial I^{2}$ of $I^{2}$ by $r$ where $r$ is the color of the $\infty$-region. Each point of the decomposition is labeled by the color of the corresponding region. Each directed edge will be labeled by $(r,a)$ where $r$ is the label of the starting point and $a$ is the color of the corresponding arc of $D_{L}$ (i.e. the label of the end point is $r*a$). For each crossing of $D_{L},$ we label the corresponding polygon in $I^{2}$ by $(r,a,b)$ where $r$ is the label of the point in the source region, $a$ is the color of the under arc adjoining the source region and $b$ is the color of the over arc. We then define a function $\psi_{X}(D_{L};\widetilde{\mathcal{C}}):(I^{2}, \partial I^{2}) \rightarrow (BX_{Q}^{X}, r)$ (or $\psi_{X}(D_{L};\widetilde{\mathcal{C}}):(I^{2}, \partial I^{2}) \rightarrow (B_{X}^{Q}X, r)$) that assigns labeled cells of $I^{2}$ to the corresponding cells in $BX_{Q}^{X}$ (or $B_{X}^{Q}X$). See Figure \ref{shadowhomotopyinvariant} for example.

\begin{figure}[h]
\centerline{{\psfig{figure=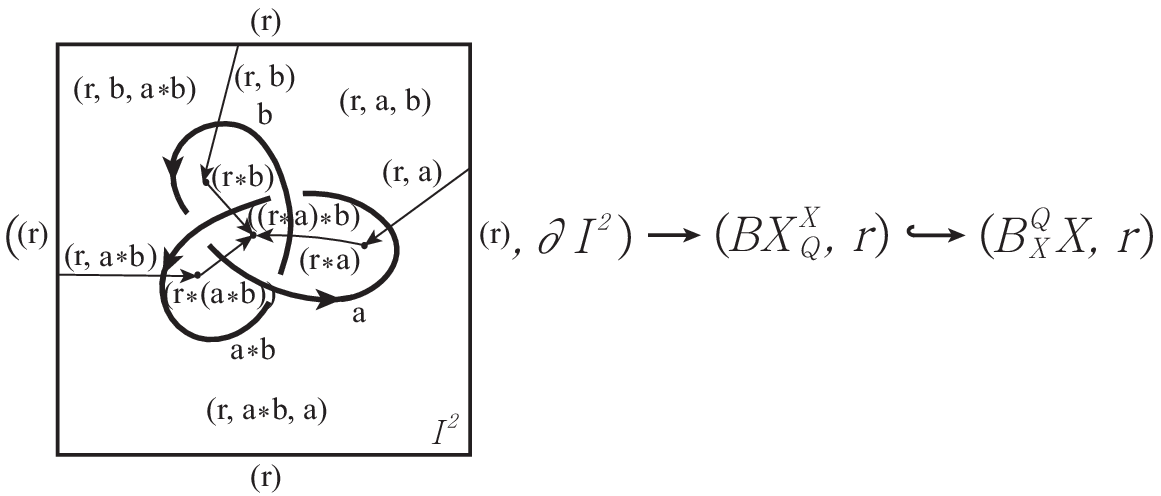,height=4.5cm}}}
\caption{A shadow homotopy invariant of an oriented knot.}
\label{shadowhomotopyinvariant}
\end{figure}

We will show in the following theorem that the homotopy class of $\psi_{X}(D_{L};\widetilde{\mathcal{C}})$ with the base point labeled by $r$ in $\pi_{2}(BX_{Q}^{X},r)$ (or $\pi_{2}(B_{X}^{Q}X,r)$) is invariant under Reidemeister moves by using the idea of Fenn, Rourke, Sanderson \cite{FenRouSan93} and Nosaka \cite{Nos11}.

\begin{theorem}
The homotopy class of $\psi_{X}(D_{L};\widetilde{\mathcal{C}})$ in $\pi_{2}(BX_{Q}^{X},r)$ (or $\pi_{2}(B_{X}^{Q}X,r)$) is invariant under Reidemeister moves.
\end{theorem}

\begin{proof}
For any $r,a \in X,$ there is a $3$-cell bounding the $2$-cell labeled by $(r,a,a)$ in the action quandle space $BX_{Q}^{X}$ (or the extended quandle space $B_{X}^{Q}X$) of $X.$ The homotopy class of $\psi_{X}(D_{L};\widetilde{\mathcal{C}}),$ therefore, is unchanged in $\pi_{2}(BX_{Q}^{X},r)$ (or $\pi_{2}(B_{X}^{Q}X,r)$), i.e. it is invariant under Reidemeister move of Type I (see Figure \ref{R1move}).

\begin{figure}[h]
\centerline{{\psfig{figure=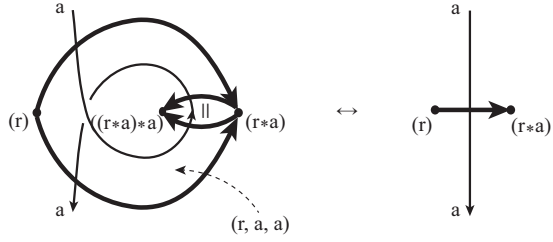,height=3.2cm}}}
\caption{The Reidemeister Type I move in $BX_{Q}^{X}$ (or $B_{X}^{Q}X$).}
\label{R1move}
\end{figure}

The homotopy class is invariant under Reidemeister move of Type II as two labeled squares corresponding to the crossings (see Figure \ref{R2move} for instance) are the same but have opposite orientation.

\begin{figure}[h]
\centerline{{\psfig{figure=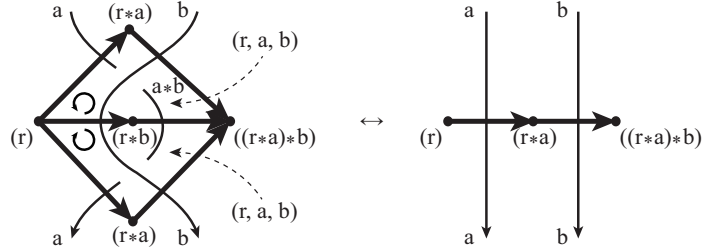,height=3.4cm}}}
\caption{The Reidemeister Type II move in $BX_{Q}^{X}$ (or $B_{X}^{Q}X$).}
\label{R2move}
\end{figure}

The union of six labeled squares which are corresponding to the six crossings in the diagrams before and after the move is the boundary of a $3$-cell in the action quandle space $BX_{Q}^{X}$ (or the extended quandle space $B_{X}^{Q}X$) of $X.$ Therefore, the homotopy class is invariant under Reidemeister move of Type III. See Figure \ref{R3move} for example.

\begin{figure}[h]
\centerline{{\psfig{figure=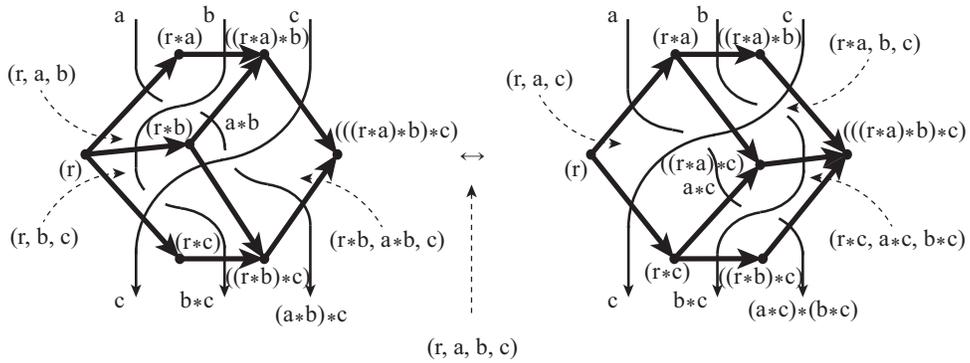,height=4.8cm}}}
\caption{The Reidemeister Type III move in $BX_{Q}^{X}$ (or $B_{X}^{Q}X$).}
\label{R3move}
\end{figure}

\end{proof}

Let $X$ be a quandle. Consider the cellular map $p:B_{X}X \rightarrow BX$ given by for any $n \geq 1,$ $p(\widetilde{e}_{(x_{1},x_{2},\ldots,x_{n+1})})=e_{(x_{2},\ldots,x_{n+1})}$ where $\widetilde{e}_{(x_{1},x_{2},\ldots,x_{n+1})}$ and $e_{(x_{2},\ldots,x_{n+1})}$ are $n$-cells labeled by $(x_{1},x_{2},\ldots,x_{n+1})$ in $B_{X}X$ and $(x_{2},\ldots,x_{n+1})$ in $BX$ respectively, and $p$ maps every $0$-cell of $B_{X}X$ to the base point of $BX.$ It is known in \cite{FenRouSan95,Nos13} that $p:B_{X}X \rightarrow BX$ and the induced map $\widetilde{p}: BX_{Q}^{X} \rightarrow B^{Q}X$ are covering maps with fibre $X.$

\begin{proposition}\cite{FenRouSan93,Nos11}
For any rack $X,$ the action of the fundamental group $\pi_{1}(BX)$ on $\pi_{2}(BX)$ is trivial. Similarly, for any quandle $X,$ the canonical action of $\pi_{1}(B^{Q}X)$ on $\pi_{2}(B^{Q}X)$ is trivial.
\end{proposition}

\begin{remark}
Notice that for a quandle $X$ the covering map $\widetilde{p}: BX_{Q}^{X} \rightarrow B^{Q}X$ induces the monomorphism $\widetilde{p}_{*}:\pi_{1}(BX_{Q}^{X},r) \rightarrow \pi_{1}(B^{Q}X,*)$ and the isomorphism $\widetilde{p}_{*}:\pi_{2}(BX_{Q}^{X},r) \rightarrow \pi_{2}(B^{Q}X,*).$ Therefore, the triviality of the action of $\pi_{1}(BX_{Q}^{X},r)$ on $\pi_{2}(BX_{Q}^{X},r)$ is inherited from $B^{Q}X.$ This justifies the omission of a base point in Definition \ref{Definition 3.4}.
\end{remark}

\begin{definition}\label{Definition 3.4}
Let $X$ be a connected quandle. If we let
$$\Psi_{X}(L)=\sum\limits_{\widetilde{\mathcal{C}} \in SCol_{X}(L)}\Psi_{X}(L;\widetilde{\mathcal{C}}) \in \mathbb{Z}[\pi_{2}(BX_{Q}^{X})],$$
where $\Psi_{X}(L;\widetilde{\mathcal{C}})$ is the homotopy class of $\psi_{X}(D_{L};\widetilde{\mathcal{C}}),$ then $\Psi_{X}(L)$ is a link invariant called the \emph{shadow homotopy invariant} of an oriented link $L.$
\end{definition}

Nosaka \cite{Nos11} defined the quandle homotopy invariant of an oriented link $L$ in a state-sum form, denoted by $\Xi_{X}(L)=\sum\limits_{\mathcal{C} \in Col_{X}(L)}\Xi_{X}(L;\mathcal{C})$ in $\mathbb{Z}[\pi_{2}(B^{Q}X)].$ Notice that for every $r \in X,$ the homotopy class of the composition $\xi_{X}(D_{L};\mathcal{C})=\widetilde{p} \circ \psi_{X}(D_{L};\widetilde{\mathcal{C}})$ coincides with $\Xi_{X}(L;\mathcal{C}).$ See the following commutative diagram:
\begin{diagram}
 & &(BX_{Q}^{X},r)\\
 & \ruTo^{\psi_{X}(D_{L};\widetilde{\mathcal{C}})} &\dTo_{\widetilde{p}}\\
(I^{2}, \partial I^{2}) &\rTo_{\xi_{X}(D_{L};\mathcal{C})} &(B^{Q}X,*).
\end{diagram}
The covering map $\widetilde{p}:(BX_{Q}^{X},r) \rightarrow (B^{Q}X,*)$ above induces the isomorphism $\widetilde{p}_{*}:\pi_{2}(BX_{Q}^{X},r) \rightarrow \pi_{2}(B^{Q}X,*)$ and $BX_{Q}^{X}$ is path-connected if $X$ is connected by Proposition \ref{Proposition 2.3}. Therefore, we have:

\begin{theorem}\label{Theorem 3.2}
Let $X$ be a finite connected quandle, and let $\Xi_{X}(L)$ be the quandle homotopy invariant of an oriented link $L.$ Then
$$\Psi_{X}(L)=|X|\Xi_{X}(L).$$
\end{theorem}

\begin{remark}
Nosaka \cite{Nos11} proved that $\pi_{2}(B^{Q}X)$ is finite if $X$ is finite and connected. Thus for a finite connected quandle $X,$ $\pi_{2}(BX_{Q}^{X})$ is a finite abelian group.
\end{remark}

A formula of the quandle homotopy invariant for the connected sum was obtained in \cite{Nos11}.

\begin{theorem}\cite{Nos11}\label{Theorem 3.4}
Let $X$ be a faithful finite connected quandle. Then for knots $K_{1}$ and $K_{2},$
$$\Xi_{X}(K_{1} \sharp K_{2})=\frac{1}{|X|}\Xi_{X}(K_{1})\Xi_{X}(K_{2})$$
where $K_{1} \sharp K_{2}$ denotes the connected sum of $K_{1}$ and $K_{2}.$
\end{theorem}

The following formula of the shadow homotopy invariant for the connected sum of two knots is immediate from Theorem \ref{Theorem 3.2} and Theorem \ref{Theorem 3.4}.

\begin{corollary}
For a faithful finite connected quandle $X,$ we have
$$|X|^{2}\Psi_{X}(K_{1} \sharp K_{2})=\Psi_{X}(K_{1})\Psi_{X}(K_{2})$$
where $K_{1}$ and $K_{2}$ are knots, and $K_{1} \sharp K_{2}$ is their connected sum.
\end{corollary}

\section{Acknowledgements}
The author would like to express his sincere appreciation to J\'ozef H. Przytycki for his friendly guidance, generous advice and constant encouragement. He is grateful to Takefumi Nosaka for valuable conversations on quandle spaces.

\ \\ \ \\
Seung Yeop Yang\\
Department of Mathematics,\\
The George Washington University,\\
Washington, DC 20052\\
e-mail: {\tt syyang@gwu.edu}

\end{document}